\documentclass[a4paper,12pt]{amsart}
\usepackage{amsmath, amssymb, amsthm, xcolor,enumerate, mathtools}
\usepackage{verbatim}
\usepackage{graphicx}
\usepackage{tikz, tikz-cd}
\usepackage{color}
\usepackage{xr}

\usepackage{verbatim}
\usepackage[colorlinks=true, citecolor=blue, linkcolor=blue, bookmarks=true]{hyperref}

\usepackage{eqparbox}
\newcommand{\eqrel}[2][default]{\mathrel{\eqmakebox[#1]{#2}}}

\newtheorem{theorem}{Theorem}[section]
\newtheorem{lemma}[theorem]{Lemma}
\newtheorem{proposition}[theorem]{Proposition}

\theoremstyle{definition}
\newtheorem{definition}[theorem]{Definition}
\newtheorem{notation}[theorem]{Notation}
\newtheorem{remark}[theorem]{Remark}

\allowdisplaybreaks

\numberwithin{equation}{section}

\makeatletter
\def\namedlabel#1#2{\begingroup
	#2
    \def\@currentlabel{#2}
    \label{#1}\endgroup
}
\makeatother

\newcommand{\Z}{\mathcal Z}
\newcommand{\F}{\mathcal F}
\newcommand{\G}{\mathcal G}
\newcommand{\C}{\mathbb C}
\newcommand{\N}{\mathbb N}
\newcommand{\id}{\mathrm{id}}
\newcommand{\M}{\mathcal{M}}
\newcommand{\eps}{\epsilon}

\newcommand{\centredrelation}[2]{{}_{\phantom{#2}}#1_{#2}}
\newcommand{\capprox}[1]{\centredrelation{\approx}{#1}}

\newcommand{\completion}[2]{\overline{#1}{}^{#2}}

\title[Property $\Gamma$ and finite dimensional boundaries]{Uniform property $\Gamma$ and finite dimensional tracial boundaries}

\author{Samuel Evington}
\address{Samuel Evington, Mathematical Institute, University of M\"unster, Ein\-stein\-strasse 62, 48149 M\"unster, Germany}
\email{evington@uni-muenster.de}

\author{Christopher Schafhauser}
\address{Christopher Schafhauser, Department of  Mathematics, University of Ne\-bra\-ska--Lincoln, 1400 R St., Lincoln, NE 68588, USA}
\email{cschafhauser2@unl.edu}

\date{\today}
\thanks{Research partially supported by: Deutsche Forschungsgemeinschaft (DFG, German Research Foundation) – Project-ID 427320536 – SFB 1442 (Evington); Germany's Excellence Strategy EXC 2044 390685587  Mathematics M{\"u}nster: Dynamics–Geometry–Structure (Evington);  ERC Advanced Grant 834267 - AMAREC (Evington); NSF grants DMS-2000129 and DMS-2400178 (Schafhauser)}

\begin{document}
\begin{abstract}
    We prove that a C$^*$-algebra $A$ has uniform property $\Gamma$ if the set of extremal tracial states, $\partial_e T(A)$, is a non-empty compact space of finite covering dimension and for each $\tau \in \partial_e T(A)$, the von Neumann algebra $\pi_\tau(A)''$ arising from the GNS representation has property $\Gamma$.
\end{abstract}
\maketitle

\section{Introduction}\label{sec:Intro}

There is a rich interaction between C$^*$-algebras and their enveloping von Neumann algebras---see Brown's survey \cite{Br11}, for example.  
More recently, this interaction has been used to great success in C$^*$-algebra theory, using Connes' fundamental result on the uniqueness of the separably acting injective II$_1$ factor (\cite{Co76}) to deduce structural theorems about simple nuclear C$^*$-algebras.  For instance, even specialising to C$^*$-algebras with unique trace, this idea played a crucial role in the solution to the Toms--Winter conjecture for C$^*$-algebras with unique trace (\cite{Wi12, MS12, MS14, SWW15}), the quasidiagonality theorem (\cite{TWW17}), and the AF embedding theorem (\cite{Schaf20}).

The condition of interest in this paper is Murray and von Neumann's property $\Gamma$ (\cite{MvN43}), which they used to show the hyperfinite II$_1$ factor is not isomorphic to a free group factor, giving the first example of non-isomorphic II$_1$ factors, by showing property $\Gamma$ holds for the former and fails for the latter. 
For our purposes, the most useful characterisation of property $\Gamma$ for a II$_1$ factor is the existence of an approximately central projection of trace 1/2, which is due to  Dixmier (\cite{Di69}).

For a C$^*$-algebra $A$, let $T(A)$ denote the set of tracial states on $A$, which we will always assume is non-empty and weak$^*$-compact (the later holds, for example, if $A$ is unital), and let $\partial_e T(A)$ denote the extreme points of $T(A)$.  For each $\tau \in T(A)$, there is an $L^2$-seminorm $\|a\|_{2,\tau} = \tau(a^*a)^{1/2}$ and the Gelfand--Naimark--Segal (GNS) representation $\pi_\tau:A \rightarrow \mathcal{B}(L^2(A,\tau))$. 

The associated tracial von Neumann algebra $\pi_\tau(A)''$ is a factor if and only if $\tau \in \partial_e T(A)$, and, in this case, $\pi_\tau(A)''$ has property $\Gamma$ if and only if for every finite set $\mathcal F \subseteq A$ and $\eps > 0$, there is a positive contraction $p \in A$ with
\begin{equation}\label{eqn:tau-gamma}
    |\tau(p) - 1/2| < \eps,\quad \|p - p^2\|_{2,\tau} < \eps,\quad \text{and}\quad  \max_{a \in \mathcal F} \|[a, p]\|_{2,\tau} < \eps.
\end{equation}
When $A$ has several traces, a natural version of property $\Gamma$, which could be called \emph{fibrewise property $\Gamma$}, would be asking that $\pi_\tau(A)''$ has property $\Gamma$ for every $\tau \in \partial_e T(A)$.  
However, the more useful condition is \emph{uniform property $\Gamma$}, introduced in \cite{CETWW}, which requires that $p$ in \eqref{eqn:tau-gamma} can be chosen uniformly over all traces $\tau \in \partial_e T(A)$.\footnote{For technical reasons, one should further require $|\tau(ap) - \tau(a)/2| < \eps$ for all $a \in \mathcal F$.  We do not know if this extra condition is automatic in general, but it is when $\partial_e T(A)$ is compact (\cite[Corollary~3.2]{CETW}), which is the case of interest in this paper.}

In \cite{CETWW}, uniform property $\Gamma$ was shown to hold for all \emph{$\mathcal Z$-stable} C$^*$-algebras $A$, i.e.\ when $A \cong A \otimes \Z$, where $\mathcal Z$ denotes the Jiang--Su algebra (\cite{JS99}). This observation had a crucial role in proof that simple nuclear finite $\mathcal Z$-stable C$^*$-algebras have nuclear dimension at most one (\cite{CETWW, CE}).\footnote{For an analogous result in the infinite case, see \cite[Theorem~G]{BBSTWW} and \cite[Theorem~7.1]{MS14}, noting that simple nulcear infinite $\mathcal Z$-stable $C^*$-algebras are purely infinite by Kirchberg's Dichotomy (\cite[Corollary~3.11(ii)]{BK}).}  Furthermore, building on work of Matui and Sato (\cite{MS12}), it was shown in \cite{CETW} that separable simple nuclear C$^*$-algebras with uniform property $\Gamma$ and strict comparison are $\mathcal Z$-stable.

Any simple nuclear non-elementary C$^*$-algebra has fibrewise property $\Gamma$ because injective II$_1$ factors have property $\Gamma$ by Connes' theorem (\cite{Co76}).
Hence, due to the results of \cite{CETW}, the problem of whether fibrewise property $\Gamma$ implies uniform property $\Gamma$ is of significant interest.  In the setting when $A$ is nuclear and $\partial_e T(A)$ is compact with finite covering dimension, this is true by the results of \cite{KR14,Sa12,TWW15}.  The main result of this article removes the nuclearity constraint.

\begin{theorem}
\label{thm:GammaFibres}
    Let $A$ be a $C^*$-algebras with $T(A)$ compact and non-empty.  Suppose $\partial_e T(A)$ is compact and has finite covering dimension.  Then $A$ has uniform property $\Gamma$ if and only if $\pi_\tau(A)''$ has property $\Gamma$ for all $\tau \in \partial_e T(A)$.
\end{theorem}

If property $\Gamma$ is replaced with McDuff's property (the existence of a unital approximately central approximate embedding of the $2 \times 2$ matrix algebra $M_2$) in both the fibrewise and uniform conditions, the theorem holds by \cite[Proposition~7.7]{KR14} and also essentially follows from the results obtained independently and contemporaneously in \cite{Sa12,TWW15}.  

Our proof of Theorem~\ref{thm:GammaFibres} is modelled on the argument in the McDuff setting carried out in \cite[Section~4]{TWW15}.  Let $A^\infty \cap A'$ denote the tracial central sequence algebra of $A$ (see Section~\ref{sec:seq-alg}).  Then the uniform McDuff property for $A$ (in the separable setting) is equivalent to the existence of a unital embedding $M_k \rightarrow A^\infty \cap A'$ for some, or equivalently any, integer $k \geq 2$.  Property $\Gamma$ is analogously characterised by the existence of unital embeddings $\mathbb C^k \rightarrow A^\infty \cap A'$ with prescribed tracial behaviour.  The extra control on the traces is not needed in the McDuff setting due to the uniqueness of the trace on $M_k$, and explicitly controlling the tracial behaviour of the maps throughout the argument is where the new difficulties lie.

It will be technically convenient to prove Theorem~\ref{thm:GammaFibres} in the slightly different (and somewhat more general) setting of W$^*$-bundles, introduced by Ozawa in \cite{Oz13}.  Since uniform property $\Gamma$ is an approximation property in the uniform 2-seminorm
\begin{equation}
    \|a\|_{2, T(A)} = \sup_{\tau \in T(A)} \tau(a^*a)^{1/2}, \qquad a \in A,
\end{equation}
it is natural to replace $A$ with it's \emph{uniform tracial completion} $\completion{A}{T(A)}$, obtained by adding a limit point to every $\|\cdot\|$-bounded $\|\cdot\|_{2, T(A)}$-Cauchy sequence in $A$ (and quotienting by $\|\cdot\|_{2, T(A)}$-null elements).  Ozawa showed in \cite{Oz13} that $\completion{A}{T(A)}$ always carries the structure of a C$^*$-algebra---in fact, these form the prototypical examples of \emph{tracially complete $C^*$-algebras}, which were recently introduced and studied  systematically in \cite{TraciallyComplete}.

When $A$ is a C$^*$-algebra such that $T(A)$ is compact and non-empty and $\partial_e T(A)$ is compact, Ozawa showed in \cite{Oz13} that the centre of $\mathcal M = \completion{A}{T(A)}$ has spectrum $K = \partial_e T(A)$ and the natural inclusion $C(K) \rightarrow \mathcal M$ admits a faithful tracial conditional expectation $E \colon \M \rightarrow C(K)$.  
Further, 
\begin{equation}
    \|a\|_{2, T(A)} = \|E(a^*a)\|^{1/2}, \qquad a \in A,
\end{equation}
and hence, by the definition of $\M = \completion{A}{T(A)}$, the unit ball of $\mathcal M$ is complete in the norm $\|b\|_{2, \rm u} = \|E(b^*b)\|^{1/2}$.  Axiomatising this structure of the triple $(\mathcal M, K, E)$ leads to Ozawa's notion of W$^*$-bundles (see Section~\ref{sec:bundles}).  Loosely speaking, $\mathcal M$ can be viewed as the continuous sections of a topological bundle over $K$ with tracial von Neumann algebra fibres.  The following is a W$^*$-bundle analogue of Theorem~\ref{thm:GammaFibres}.

\begin{theorem}
\label{thm:GammaFibresBundles}
Let $\M$ be a $W^*$-bundle over a finite dimensional compact Hausdorff space such that every fibre of $\M$ is a {\rm II}$_1$ factor.  Then $\M$ has property $\Gamma$ if and only if every fibre of $\M$ has property $\Gamma$.
\end{theorem}

If $A$ is a C$^*$-algebra as in Theorem~\ref{thm:GammaFibres}, the corresponding W$^*$-bundle $\M = \completion{A}{T(A)}$ will satisfy the hypotheses of
Theorem~\ref{thm:GammaFibresBundles}.  
Then property $\Gamma$ for $\M$, coming from Theorem~\ref{thm:GammaFibresBundles}, will imply uniform property $\Gamma$ for $A$, 
obtaining Theorem~\ref{thm:GammaFibres}.  
After establishing some preliminaries in Section~\ref{sec:Prelims},  the rest of the paper is essentially devoted to proving Theorem~\ref{thm:GammaFibresBundles} in Section~\ref{sec:main}. Theorem~\ref{thm:GammaFibres} is deduced from Theorem~\ref{thm:GammaFibresBundles} at the end of Section~\ref{sec:main}.

\subsection*{Acknowledgements}

This work grew out of the authors' joint work with Jos\'e Carri{\'o}n, Jorge Castillejos, Jamie Gabe, Aaron Tikuisis, and Stuart White in \cite{TraciallyComplete}, and we thank them for several discussions regarding this work and related results.  We also thank Ilijas Farah for his helpful comments on an earlier draft of this paper.  Finally, we thank the referee for their suggestions.

\section{Preliminaries}\label{sec:Prelims}

\subsection{\texorpdfstring{W$^*$}{W*}-bundles}\label{sec:bundles}

W$^*$-bundles will be central to this paper. We recall the definition and set out our notational conventions below. Our standard references for W$^*$-bundles are \cite{Oz13} and \cite{Ev18}.
\begin{definition}[{cf.\ \cite[Section 5]{Oz13}}]
    A \emph{$W^*$-bundle} consists of a unital C$^*$-algebra $\M$ together with a unital embedding of $C(K)$ into the centre of $\M$ and a conditional expectation $E\colon \M \rightarrow C(K)$ such that the following axioms hold:
    \begin{enumerate}[(i)] \label{def:WstarBundle}
        \item for any $a,b \in \M$, we have $E(ab) = E(ba)$;
        \item for any $a \in \M$, we have $E(a^*a) = 0$ implies $a = 0$;
        \item the unit ball $\{a \in \M: \|a\| \leq 1\}$ is complete with respect to the norm defined by  $\|a\|_{2,\rm u} = \|E(a^*a)\|^{1/2}$.
    \end{enumerate}
\end{definition}
We shall denote a W$^*$-bundle by a triple $(\M,K,E)$ or simply by $\M$ if $K$ and $E$ are clear from context. Every point $x \in K$ defines a trace $\tau \in T(\M)$ by $\tau(a) = E(a)(x)$ for $a \in \M$.  The map $K \rightarrow T(\M)$ thus defined is continuous with respect to the weak$^*$ topology on $T(\M)$. It will be convenient to identify points in $K$ with their induced trace. 

We write $\pi_\tau \colon \M \rightarrow \mathcal B(L^2(\M, \tau))$ for the GNS representation of $\M$ with respect to $\tau \in K$. The image $\pi_\tau(\M)$ is called the \emph{fibre} of $\M$ at $\tau \in K$. An important consequence of axiom (iii) is that $\pi_\tau(\M)'' = \pi_\tau(\M)$; see \cite[Theorem 11]{Oz13}. The trace $\tau$ induces a faithful normal trace $\bar{\tau}$ on $\pi_\tau(\M)$, so the fibres of a W$^*$-bundle are tracial von Neumann algebras.  A W$^*$-bundle is said to have \emph{factorial fibres} if $\pi_\tau(\M)$ is a factor for all $\tau \in K$. This is equivalent to saying that every $\tau \in K$ is an extreme point of $T(\M)$; see \cite[Theorem 6.7.3]{Di77}, for example.

Given a C$^*$-algebra $A$ with $T(A)$ compact and non-empty, the \emph{uniform tracial completion} of $A$ with respect to $T(A)$ is defined by 
\begin{equation}
    \completion{A}{T(A)} = \frac{\{(a_n)_{n=1}^\infty \in \ell^\infty(A): (a_n)_{n=1}^\infty \text{ is }\|\cdot\|_{2,T(A)}\text{-Cauchy}\}}{\{(a_n)_{n=1}^\infty \in \ell^\infty(A): (a_n)_{n=1}^\infty \text{ is }\|\cdot\|_{2,T(A)}\text{-null}\}},
\end{equation}
where $\ell^\infty(A)$ denotes the C$^*$-algebra of bounded sequences in $A$ and 
\begin{equation}
    \|a\|_{2,T(A)} = \sup_{\tau \in T(A)} \|a\|_{2,\tau}, \qquad a \in A.
\end{equation} 
Ozawa proved that for such a C$^*$-algebra, if the set of extreme points of $T(A)$, denoted $\partial_e T(A)$, is compact in the weak$^*$ topology, then $\completion{A}{T(A)}$ can be endowed with the structure of a W$^*$-bundle over $K = \partial_e T(A)$; see \cite[Theorem 3]{Oz13}. 

W$^*$-bundles form a special case of the more general framework of tracially complete C$^*$-algebras recently introduced in \cite{TraciallyComplete}. A \emph{tracially complete $C^*$-algebra} is a pair $(\M, X)$, where $\M$ is a unital C$^*$-algebra and $X \subseteq T(\M)$ is a compact convex set of traces, where the seminorm \begin{equation}\label{eq:trace-norm}
    \|a\|_{2,X} = \sup_{\tau \in X} \|a\|_{2,\tau}, \qquad a \in A, 
\end{equation} 
is a norm and the unit ball $\{a \in \mathcal M : \|a\| \leq 1\}$ is $\|\cdot\|_{2,X}$-complete.  More precisely, given a W$^*$-bundle $(\mathcal M, K, E)$, let $X$ be the set of all traces of the form
\begin{equation}
    \tau_\mu(a) = \int_K E(a) \, d\mu, \qquad a \in \mathcal M,
\end{equation}
where $\mu$ ranges over the space of Radon probability measures on $K$.  Then $(\M, X)$ is a tracially complete C$^*$-algebra (\cite[Proposition~3.6]{TraciallyComplete}).  
By a theorem essentially due to Ozawa in \cite{Oz13}, W$^*$-bundles with factorial fibres are precisely the factorial\footnote{We recall from \cite[Definition~3.13]{TraciallyComplete} that a tracially complete C$^*$-algebra $(\mathcal M, X)$ is \emph{factorial} 
if $X$ is a face in $T(\mathcal M)$.  This happens precisely when $\pi_\tau(\mathcal M)''$ is a factor for all $\tau \in \partial_e X$ (see \cite[Proposition~3.14]{TraciallyComplete}), and hence factoriality for tracially complete C$^*$-algebras generalises the notion of factorial fibres for W$^*$-bundles.} tracially complete C$^*$-algebras $(\M, X)$ where $X$ is a Bauer simplex; this precise statement is given as \cite[Theorem 3.37]{TraciallyComplete}.  In fact, tracially complete C$^*$-algebras were introduced to extend the thoery of W$^*$-bundles beyond the Bauer setting.  For example, if $A$ is a unital C$^*$-algebra such that $T(A)$ is compact and non-empty but $\partial_e T(A)$ is not compact, then $\completion{A}{T(A)}$ does not have a natural W$^*$-bundle structure but is still a (factorial) tracially complete C$^*$-algebra.

\subsection{Sequence algebras}\label{sec:seq-alg}

Ultrapowers of W$^*$-bundles were introduced in \cite[Section 3]{BBSTWW}. In this paper, it will be more convenient to work with the Fr{\'e}chet filter on $\N$ rather than an ultrafilter; i.e.\ we will work with classical sequential limits instead of ultralimits.

\begin{definition}\label{def:seq-algebra}
Let $(\M,K,E)$ be a W$^*$-bundle. Then 
\begin{equation}
    c_{0, \rm u}(\M) = \{(a_n)_{n=1}^\infty \in \ell^\infty(\M): \lim_{n\to\infty}\|a_n\|_{2,\rm u} = 0 \}
\end{equation} 
is an ideal of the C$^*$-algebra $\ell^\infty(\M)$ of $\|\cdot\|$-bounded sequences in $\M$, and we define $\M^\infty = \ell^\infty(\M)/c_{0,\rm u}(\M)$.
Since $\|f\|_{2,\rm u} = \|f\|$ for all $f \in C(K)$, the norm sequence algebra\footnote{It has become common to use subscripts such as $A_\infty$ for C$^*$-norm sequence algebras and superscripts such as $\mathcal M^\infty$ for (uniform) tracial sequence algebras.  Since the only C$^*$-norm sequence algebra appearing in this paper is $C(K)^\infty$ (and the C$^*$-norm on $C(K)$ agrees with the uniform trace norm $\|\cdot\|_{2, T(C(K))}$), there should be no ambiguity caused by using the notation $C(K)^\infty$ for the C$^*$-norm sequence algebra.} 
\begin{equation}
    C(K)^\infty = \frac{\ell^\infty(C(K))}{\{(f_n)_{n=1}^\infty \in \ell^\infty(C(K)): \lim_{n \rightarrow \infty} \|f_n\| = 0 \}}.
\end{equation}
unitally embeds into the centre of $\M^\infty$.
There is a conditional expectation $E^\infty \colon \M^\infty \rightarrow C(K)^\infty$, defined at the level of representative sequences by $(a_n)_{n=1}^\infty \mapsto (E(a_n))_{n=1}^\infty$. 
We write $K^\infty$ for the spectrum of the abelian C$^*$-algebra $C(K)^\infty$ and identify $C(K)^\infty \cong C(K^\infty)$.
The W$^*$-bundle $(\M^\infty, K^\infty, E^\infty)$ is called the \emph{sequence algebra} of $\M$.\footnote{The only difficult part of showing that $(M^\infty, K^\infty, E^\infty)$ is a W$^*$-bundle is proving $\|\cdot\|_{2,\rm u}$-completeness of the unit ball. This is achieved using Kirchberg's $\eps$-test; see \cite[Proposition 3.9]{BBSTWW} or \cite[Proposition 5.4]{TraciallyComplete}, for example.} 
\end{definition}

It is worth saying a few extra words about the base space $K^\infty$ of the reduced power. 
For every sequence of points $(x_n)_{n=1}^\infty$ in $K$ and every free ultrafilter $\omega$ on the natural numbers, we can define a character $x_\omega \colon C(K)^\infty \rightarrow\C$ by $(f_n)_{n=1}^\infty \mapsto \lim_{n\to\omega} f_n(x_n)$. 
Hence, $x_\omega \in K^\infty$. 
The set of all such characters recovers the norm on $C(K)^\infty$ and so defines a dense subset of $K^\infty$ by a standard application of Urysohn's lemma.  
When we view elements of $K^\infty$ as traces on $\M^\infty$, by identifying $x_\omega$ with $x_\omega \circ E^\infty$, the characters of the form $x_\omega$ correspond to limit traces in the sense of \cite[Section~1]{CETW-classification} (see also \cite[Section~1.3]{BBSTWW} for an ultrapower version). Hence, we may view $K^\infty$ as a subset of the weak$^*$-closure of the limit traces on $\mathcal M^\infty$.

We end this subsection  by reminding the reader of some common notational conventions. We identify $\M$ with the subalgebra of $\M^\infty$ coming from constant sequence in $\ell^\infty(\M)$ and write $\M^\infty \cap S'$ for the relative commutant of a subset $S \subseteq \M^\infty$.

\subsection{Tracial factorisation}
W$^*$-bundles with factorial fibres enjoy the following property known as \emph{tracial factorisation}.

\begin{proposition}\label{prop:tf-F-epsilon}
	Let $(\M,K,E)$ be a $W^*$-bundle with factorial fibres. For any finite subset of contractions $\F \subseteq \M$ and $\eps > 0$, there exist a finite subset of positive contractions $\G \subseteq \M$ and $\delta > 0$ such that for all $x \in \M$, if
\begin{equation}
	\max_{y \in \G} \|[x,y]\|_{2,\rm u} < \delta,
\end{equation}
then
\begin{equation}
	\max_{y \in \F} \|E(xy)-E(x)E(y)\| < \eps.
\end{equation} 
\end{proposition}

The fact that W$^*$-bundles with factorial fibres have tracial factorisation is implicit in \cite{Oz13}. Essentially the same phenomenon, in the setting of C$^*$-algebras with a Bauer simplex of traces, is shown in \cite{Sa12}. 
Our proof of Proposition \ref{prop:tf-F-epsilon} is modelled on \cite[Proposition 3.1]{CETW} with nets replacing sequences.

\begin{proof}[Proof of Proposition~\ref{prop:tf-F-epsilon}]
Suppose the result doesn't hold.  Then there exist $\eps_0 > 0$, a positive contraction $y_0 \in \M$, and a net $(x_\lambda)_{\lambda \in \Lambda}$ of positive contractions in $\M$ such that 
\begin{align}\label{eq:tf-commute}
	\lim_{\lambda} \|[x_\lambda,y]\|_{2,\rm u} = 0
\end{align} 	 
for all $y \in \M$, but
\begin{align}
	\|E(x_\lambda y_0)-E(x_\lambda)E(y_0)\|  \geq \eps_0
\end{align} 
for all $\lambda \in \Lambda$.  Hence, there exists a net $(\tau_\lambda)_{\lambda \in \Lambda}$ of traces in $K$ such that
\begin{align}\label{tf-not-muliplicative}
	|\tau_\lambda(x_\lambda y_0)-\tau_\lambda(x_\lambda)\tau_\lambda(y_0)| \geq \eps_0
\end{align} 
for each $\lambda \in \Lambda$.

Since $K$ is compact, after passing to a subnet, we may assume $(\tau_\lambda)_{\lambda \in \Lambda}$ converges in the weak$^*$ topology to some $\tau \in K$. Since the unit ball of $\M^*$ is weak$^*$-compact, by passing to a subnet again, we may further assume that $(y \mapsto \tau_{\lambda}(x_{\lambda} y))_{\lambda \in \Lambda}$ converges in the weak$^*$ topology to some $\sigma \in \M^*$. 

It follows from \eqref{eq:tf-commute} that $\sigma$ is a positive tracial functional on $\M$. Moreover, for positive $y \in \M$, we have
\begin{equation}
\begin{split}
    \sigma(y) &= \lim_{\lambda} \tau_{\lambda}(y^{1/2} x_{\lambda} y^{1/2})\\
	&\leq \limsup_{\lambda} \tau_{\lambda}(y^{1/2} \|x_{\lambda}\| y^{1/2})\\
	&\leq \tau(y)
\end{split}
\end{equation}
since $x_\lambda \in \M$ is a positive contraction. 

As $\M$ has factorial fibres, $\tau$ is an extremal trace on $\M$. Since $\sigma \leq \tau$, it follows that $\sigma = \sigma(1_\mathcal M) \tau$. We conclude 
\begin{equation}
\begin{split}
	\lim_{\lambda} \tau_{\lambda}(x_{\lambda} y_0) 
	&= \sigma(y_0) \\
	&= \sigma(1_\mathcal M) \tau(y_0)  \\
	&= \lim_{\lambda} \tau_{\lambda}(x_{\lambda})\lim_{\lambda} \tau_{\lambda}(y_0) \\
	&= \lim_{\lambda} \tau_{\lambda}(x_{\lambda})\tau_{\lambda}(y_0).
\end{split}
\end{equation}
However, this contradicts \eqref{tf-not-muliplicative}.
\end{proof}

In this paper, we will make judicious use of tracial factorisation to expedite our proofs in the following way. We shall show that elements with certain properties exist in all relative commutants $\M^\infty \cap \M_0'$, where $\M_0$ is any $\|\cdot\|_{2,\rm u}$-separable subalgebra of a W$^*$-bundle $\M$.
A reindexing argument will then allow us to show that elements with the same set of properties exists in $\M^\infty \cap S'$ for any $\|\cdot\|_{2,\rm u}$-separable subalgebra $S \subseteq \M^\infty$, and moreover, the elements can be chosen such that each element $a$ satisfies $\tau(as) = \tau(a)\tau(s)$  for all $s \in S$ and $\tau \in K^\infty$.

A formal statement and proof of this fact will be presented in the following lemma. To this end, we introduce some additional terminology:
the \emph{reindexing $^*$-homomorphism} $\psi_\rho \colon \M^\infty \rightarrow \M^\infty$ associated to a strictly increasing function $\rho \colon \mathbb N \rightarrow \mathbb N$ is the unital $^*$-homomorphism defined at the level of representative sequences by $(a_n)_{n=1}^\infty \mapsto (a_{\rho(n)})_{n=1}^\infty$.\footnote{The subtlety here is that $\psi_\rho$ is well-defined. This is true for the sequence algebra as $\lim_{n\to\infty}\|a_{\rho(n)}\|_{2,\rm u} = 0$ whenever $\lim_{n\to\infty}\|a_n\|_{2,\rm u} = 0$, but it is not always true for ultrapowers. This is the reason for our choice to work with sequential limits.}

\begin{lemma}\label{lem:ReindexTracialfactorisation}
Let $(\M,K,E)$ be a $W^*$-bundle with factorial fibres. 
For a $\|\cdot\|_{2, \rm u}$-separable subset $S \subseteq \M^\infty$, there is a  $\|\cdot\|_{2,\rm u}$-separable subalgebra $\M_0 \subseteq \M$ with the following property:
for any  $\|\cdot\|_{2,\rm u}$-separable subset $T \subseteq \M^\infty \cap \M_0'$, there exists a reindexing $^*$-homomorphism $\psi_\rho \colon \M^\infty \rightarrow \M^\infty$ such that $\psi_\rho(T) \subseteq \M^\infty \cap S'$ and
\begin{equation}
\label{eq:ReindexTracialfactorisation}
 \tau(\psi_\rho(t)s)=\tau(\psi_\rho(t))\tau(s) \end{equation}
for all $s \in S$, $t\in T$, and $\tau \in K^\infty$.
\end{lemma}

\begin{proof}
By continuity and linearity, it suffices to replace $S$ by a countable set of contractions.
Let us then enumerate $S=\{s^{(1)},s^{(2)},\dots\}$ and represent $s^{(i)}$ by the sequence of contractions $(s_n^{(i)})_{n=1}^\infty$ in $\M$.
Set $\F_n = \{s_n^{(i)}: i = 1,\ldots,n\}$ and $\eps_n = \tfrac{1}{n}$. Let $\G_n \subseteq \M$ and $\delta_n > 0$ be the finite set and tolerance corresponding to $(\F_n, \eps_n)$ according to Proposition~\ref{prop:tf-F-epsilon}. We may assume that $\F_n \subseteq \G_n$ and $\delta_n < \eps_n$. Take $\M_0$ to be the subalgebra of $\M$ generated by $\bigcup_{n \in \N} \G_n$ and note that $\M_0$ is $\|\cdot\|_{2,\rm u}$-separable.

Let $T \subseteq \M^\infty \cap \M_0'$ be $\|\cdot\|_{2,\rm u}$-separable. By continuity and linearity, it suffices to replace $T$ by a countable set of contractions. Say $T=\{t^{(1)},t^{(2)},\dots\}$ and represent $t^{(j)}$ by the sequence of contractions $(t_m^{(j)})_{m=1}^\infty$.
For each $n \in \N$, any sufficiently large $m \in \N$ will satisfy
\begin{equation}
	\max_{y \in \G_n}\|[t_{m}^{(j)}, y]\|_{2,\rm u} < \delta_n 
\end{equation}
for all $j\in\{1, \ldots, n\}$ because $T$ commutes with $\M_0$. Hence, we may inductively define a strictly increasing function $\rho\colon\N \rightarrow \N$ such that
\begin{equation}\label{eq:eta-commute}
	\max_{y \in \G_n}\|[t_{\rho(n)}^{(j)}, y]\|_{2,\rm u} < \delta_n 
\end{equation}
for all $j\in\{1, \ldots, n\}$ and $n \in \mathbb N$.  By the choice of $\mathcal G_n$ and $\delta_n$, this implies 
\begin{align}
	\sup_{\tau \in K}\big|\tau(t_{\rho(n)}^{(j)}s_{n}^{(i)})-\tau(t_{\rho(n)}^{(j)})\tau(s_{n}^{(i)})\big| &< \frac{1}{n}
\end{align}
for all $i, j \in\{1, \ldots, n\}$ and $n \in \mathbb N$.  At the level of the sequence algebra $\M^\infty$, this implies \eqref{eq:ReindexTracialfactorisation}. Since we have chosen that $\F_n \subseteq \G_n$ and $\delta_n < \eps_n =\tfrac{1}{n} $, it follows from \eqref{eq:eta-commute} that $\psi_\rho(T) \subseteq \M^\infty \cap S'$. 
\end{proof}

Note that the formulation of Lemma \ref{lem:ReindexTracialfactorisation} simplifies when $\M$ itself is $\|\cdot\|_{2,\rm u}$-separable as $\M_0$ can always be taken to be $\M$.

\subsection{Property \texorpdfstring{$\Gamma$}{gamma}}\label{subsec:property-gamma}

Uniform property $\Gamma$ for C$^*$-algebras was introduced in \cite{CETWW} and further investigated in \cite{CETW}.

\begin{definition}[{\cite[Definition 2.1]{CETWW}}]\label{def:uniform-gamma}
Let $A$ be a C$^*$-algebra with $T(A)$ non-empty and compact. Then $A$ has \emph{uniform property $\Gamma$} if for any separable subset $S \subseteq A$ and $k \in \N$, there exist projections $p_1,\ldots,p_k \in A^\infty \cap S'$ summing to $1_{A^\infty}$ such that
\begin{equation}\label{eq:tracially-divides}
  \tau(ap_j) = \frac{1}{k}\tau(a)
\end{equation}
 for all $a \in S$, $\tau \in T_\infty(A)$ and $j\in\{1,\dots,k\}$.
\end{definition}

Here, the \emph{tracial sequence algebra} $A^\infty$ is defined analogously to the sequence algebra $\mathcal M^\infty$ in Definition~\ref{def:seq-algebra}, replacing with uniform trace norm $\|\cdot\|_{2, \rm u}$ on the W$^*$-bundle $\M$ with the uniform trace seminorm $\|\cdot\|_{2, T(A)}$ as in \eqref{eq:trace-norm}.  The set $T_\infty(A) \subseteq T(A^\infty)$ is the set of limit traces, defined by on representing sequences by
\begin{equation} 
    (a_n)_{n=1}^\infty \mapsto \lim_{n \rightarrow \omega} \tau_n(a_n)
\end{equation}
for a sequence of traces $(\tau_n)_{n=1}^\infty \subseteq T(A)$ and a free ultrafilter $\omega$ on $\mathbb N$. 

The definition in \cite{CETWW} differs in two ways.
When $A$ itself is separable, it suffices to take $S = A$ by a simple reindexing argument.  Also, \cite{CETWW} works with the ultrapower $A^\omega$ in place of the sequence algebra $A^\infty$.  Both constructions lead to the same notion of property $\Gamma$; see the  discussion in \cite[Section 2]{CETW} for details.  Finally, as discussed in \cite[Section 3]{CETW}, when $T(A)$ is a Bauer simplex it suffices take $a=1_A$ in this definition by tracial factorisation.  Also, we note that the informal definition of property $\Gamma$ stated in the introduction assumed $k = 2$.  This is the same condition by a slight modification of \cite[Proposition~2.3]{CETW}.\footnote{The statement in \cite[Proposition~2.3]{CETW} assumes separability, but this is not hard to remove from the proof after replacing $A^\infty \cap A'$ with $A^\infty \cap S'$ for separable $S \subseteq A$.}  However, it is important in this paper that $k$ can be taken to be arbitrarily large when applying property $\Gamma$, and when verifying property $\Gamma$, reducing to the case $k = 2$ will not significantly simplify our proof.

Property $\Gamma$ can also be defined at the level of W$^*$-bundles with factorial fibres. 
\begin{definition}\label{def:gamma-bundle}
Let $(\M,K,E)$ be a W$^*$-bundle with factorial fibres.  We say that $(\M,K,E)$ has \emph{property $\Gamma$} if for any $\|\cdot\|_{2,\rm u}$-separable subset $S \subseteq \M$ and $k \in \N$, there exist projections $p_1,\ldots,p_k \in \M^\infty \cap S'$ summing to $1_{\M^\infty}$ such that
\begin{equation}\label{eq:tracially-divides-Wstar}
  \tau(p_j) = \frac{1}{k}
\end{equation}
 for all $\tau \in K^\infty$ and $j\in\{1,\dots,k\}$.
\end{definition}
By tracial factorisation, the $p_j$ can always be chosen such that   
\begin{equation}\label{eq:tracially-divides-Wstar-2}
  \tau(ap_j) = \frac{1}{k}\tau(a)
\end{equation}
for all $a \in S$, $\tau \in K^\infty$ and $j\in\{1,\dots,k\}$.
Hence, our definition of property $\Gamma$ for W$^*$-bundles is consistent with that of \cite[Definition~5.19]{TraciallyComplete} for tracially complete C$^*$-algebras.

Finally, we note that the two notions of property $\Gamma$ above are closely connected, which is what will allow us to deduce Theorem~\ref{thm:GammaFibres} from Theorem~\ref{thm:GammaFibresBundles}.  If $A$ is a C$^*$-algebra with $T(A)$ compact and non-empty and $K = \partial_e T(A)$ compact, then $\M=\completion{A}{T(A)}$ is a W$^*$-bundle as recalled in Section~\ref{sec:bundles}, and $A$ has uniform property $\Gamma$ if and only if $\M$ has property $\Gamma$; see \cite[Proposition~5.20]{TraciallyComplete}.

\subsection{Order zero maps and their functional calculus}\label{subsec:order-zero}

Let $A$ and $B$ be C$^*$-algebras. A \emph{completely positive and contractive} (c.p.c.) map $\phi\colon A \rightarrow B$ is said to be \emph{order zero} if it preserves orthogonality, i.e.\ if $\phi(x)\phi(y) = 0$ for all $x,y \in A_+$ satisfying $xy = 0$. 

We briefly recall the structure theorem for order zero maps and the order zero functional calculus from \cite{WZ09}, which is based on early work from \cite{W}.  Let $\phi \colon A \rightarrow B$ be a c.p.c.\ order zero map. Then there is positive contraction $h \in M({\rm C}^*(\phi(A))) \cap \phi(A)'$ and a $^*$-homomorphism $\hat \phi \colon A \rightarrow M({\rm C}^*(\phi(A))) \cap \{h\}'$ such that $\phi(a) = \hat\phi(a)h$ for all $a \in A$ (see \cite[Theorem~2.3]{WZ09}). Note that when $A$ is unital, we have $h=\phi(1_A) \in {\rm C}^*(\phi(A))$.

For a positive contraction $f \in C_0(0, 1]$, define $f(\phi) \colon A \rightarrow B$ by $f(\phi)(a) = \hat\phi(a) f(h)$. 
Since  $\hat\phi$ is a $^*$-homomorphism commuting with $h$, it is easily seen that $f(\phi)$ is a c.p.c.\ order zero map. 
This construction is known as the \emph{order zero functional calculus}.
Furthermore, we can define an induced $^*$-homomorphism $\tilde\phi\colon C_0(0,1] \otimes A \rightarrow B$ via $f \otimes a \mapsto f(\phi)(a)$ (see \cite[Corollary 3.1]{WZ09}). The original c.p.c.\ order zero map $\phi$ can be recovered from $\tilde\phi$ since $\phi(a) = \tilde\phi(\id_{(0,1]} \otimes a)$ for all $a \in A$. 

The following property of the order zero functional calculus is particularly relevant to the computations in this paper.  This has been observed before (see \cite[(2.1)]{TWW15}, for example), but we include a proof for the sake of completeness.

\begin{lemma}\label{lem:oz-projection}
    Let $A$ and $B$ be $C^*$-algebras and $\phi\colon A \rightarrow B$ be a c.p.c.\ order zero map. Let $p \in A$ be a projection. Then
    \begin{equation}
        f(\phi)(p) = f(\phi(p))
    \end{equation}
     for all positive contractions $f \in C_0(0,1]$.
\end{lemma}
\begin{proof}
    By the Stone--Weierstrass theorem, it suffices to consider the functions $f_n(t) = t^n$ for $n \geq 1$.  
    Let $\tilde\phi\colon C_0(0,1] \otimes A \rightarrow B$ be the induced $^*$-homomorphism. 
    Then, since $p^n = p$, we have
    \begin{equation}
        f_n(\phi(p)) = \tilde\phi(\id_{(0,1]} \otimes p)^n = \tilde\phi(\id_{(0,1]}^n \otimes p) = f_n(\phi)(p). \qedhere
    \end{equation}
\end{proof}

We isolate the following lemma from the proof of \cite[Lemma~4.5]{TWW15}.  We thank Allan Donsig for suggesting the short spatial proof below, which we find more intuitive than the functional calculus approach taken in \cite{TWW15}.

\begin{lemma}\label{lem:oz-leq}
   Let $A$ and $B$ be $C^*$-algebras and assume $\phi, \psi \colon A \rightarrow B$ are c.p.c.\ maps with $ \phi(a) \leq \psi(a)$ for all positive $a \in A$.  If $\psi$ is order zero, then so is $\phi$.
\end{lemma}

\begin{proof}
    Fix a faithful representation $B \subseteq \mathcal B(\mathcal H)$.  If $a, b \in A$ are positive with $ab = 0$, then $\psi(a) \psi(b) = 0$.  Combining this with the inequalities  $0 \leq \phi(a) \leq \psi(a)$ and $0 \leq \phi(b) \leq \psi(b)$ yields
    \begin{equation}
        \overline{\phi(b)\mathcal H} \subseteq \overline{\psi(b)\mathcal H} \subseteq \ker \psi(a) \subseteq \ker \phi(a),
    \end{equation}
    and hence $\phi(a) \phi(b) = 0$.
\end{proof}

\section{Finite covering dimension and property \texorpdfstring{$\Gamma$}{gamma}}\label{sec:main}

We now begin our journey towards Theorem \ref{thm:GammaFibresBundles}. As the argument is fairly technical, we have broken it down into a series of lemmas, each presented in its own subsection together with some additional commentary. 
For convenience, we shall make the following global notational conventions. 

\begin{notation}
We write $e_1,\dots,e_k$ for the minimal projections of $\C^k$ and $1_k$ for the unit of $\C^k$. 
For $z_1, z_2 \in \C$ and $\eps > 0$, we write $z_1 \approx_\eps z_2$ as a shorthand for $|z_1 - z_2| \leq \eps$. 
For positive elements $x$ and $y$ of a C$^*$-algebra, we write $x \perp y$ if $xy = yx = 0$.
For future use in functional calculus, we define the continuous functions $g_{\gamma_1,\gamma_2}\in C_0(0,1]$,  where $0\leq\gamma_1<\gamma_2\leq 1$, by
\begin{equation}
\label{eq:gDef}
	g_{\gamma_1, \gamma_2}(t) = \begin{cases} 0, & 0 < t \leq \gamma_1; \\ \frac{t - \gamma_1}{\gamma_2 - \gamma_1}, & \gamma_1 < t < \gamma_2; \\ 1, & \gamma_2 \leq t \leq 1. \end{cases} 
\end{equation}
\end{notation}

Our construction makes systematic use of c.p.c.\ order zero maps $\phi\colon \C^k \rightarrow A$, where $A$ is a C$^*$-algebra of interest. 
The reader is encouraged to think of these objects as a convenient packaging for $k$ mutually orthogonal positive contractions in the C$^*$-algebra $A$, namely the elements $\phi(e_1),\ldots,\phi(e_k) \in A$. Note that $\phi$ being a $^*$-homomorphism is equivalent to these positive elements being projections.

\subsection{The core partition of unity argument}

The first step in the proof of Theorem \ref{thm:GammaFibresBundles} is to apply a standard W$^*$-bundle partition of unity results to the property that all the fibres are II$_1$ factors with property $\Gamma$.  In each fibre, we have $k$ orthogonal and approximately central projections, each of trace $1/k$, that sum to the identity.  Gluing them together over a partition of unity results in approximately central positive contractions $a_1, \ldots, a_k \in \mathcal M$, each of trace approximately $1/k$, that sum to the identify.   

Using that $m = \dim(K) < \infty$, we can chose the aforementioned partition of unity so that any point in $K$ is contained in the support of at most $m+1$ of the functions in the partition.  This allows us to decompose each $a_j$ as a sum $\sum_{c = 0}^m a_j^{(c)}$,
where each $a_j^{(c)}$ is a positive contraction and, for each $c \in\{0,\ldots,m\}$, the elements $a_1^{(c)}, \ldots, a_k^{(c)}$ are mutually orthogonal, 
providing an `$(m+1)$-coloured' version of property $\Gamma$.  

It is crucial in the rest of the argument that $m$ can be chosen uniformly over $k \in \mathbb N$ and over all tolerances used to measure approximate centrality (i.e.\ that $m$ does not depend on the separable set $S$).  This is where the finiteness of $\dim(K)$ enters the proof of Theorem~\ref{thm:GammaFibresBundles}.

\begin{lemma}[{cf.\ \cite[Lemma 4.1]{TWW15}}]
\label{lem:GammaFibresMain}
 Let $m \in \N$ and let $(\M,K,E)$ be a $W^*$-bundle whose fibres are {\rm II}$_1$ factors with property $\Gamma$ and such that $\dim(K) \leq m$.  Further, let $S \subseteq \M$ be a  $\|\cdot\|_{2,\rm u}$-separable subset. 
Then for every $k \in \N$ there exist c.p.c.\ order zero maps $\Phi^{(0)},\dots,\Phi^{(m)}\colon\C^k \to \M^\infty \cap S'$ such that
\begin{align}
\label{eq:GammaFibresMainA}
\sum_{c'=0}^m \Phi^{(c')}(1_k) &= 1_{\M^\infty}
\shortintertext{and}
\label{eq:GammaFibresMainB}
\tau(f(\Phi^{(c)}(e_j)))&=\frac1k\tau(f(\Phi^{(c)}(1_k)))
\end{align}
 for all $\tau\in K^\infty,\ c\in\{0,\dots,m\},\ j\in\{1,\dots,k\}$, and $f \in C_0(0,1]$.
\end{lemma}

\begin{proof}
Unpacking the sequence algebra formalism, it suffices to show that for any finite set $\mathcal F \subseteq \M$, $N\in\N$, and $\eps>0$, there are c.p.c.\ order zero maps $\Phi^{(0)},\dots,\Phi^{(m)}\colon \C^k \to \M$ such that
\begin{align}
\label{eq:GammaFibresMain2}
\Big\|\sum_{c'=0}^m \Phi^{(c')}(1_k)-1_{\M}\Big\|_{2,\rm u} &\leq \eps, \\
\label{eq:GammaFibresMain1}
\big\|[\Phi^{(c)}(e_j),b]\big\|_{2,\rm u} &\leq \eps,
\shortintertext{and}
\label{eq:GammaFibresMain3}
\big|\tau(\Phi^{(c)}(e_j)^n)-\frac1k\tau(\Phi^{(c)}(1_k)^n)\big| &\leq \eps
\end{align}
for all $\tau \in K$, $c\in\{0,\dots,m\}$,  $j\in\{1,\dots,k\}$, $n\in\{1,\dots,N\}$, and $b \in \mathcal F$.

For every $\tau \in K$, since $\pi_\tau(\M)$ has property $\Gamma$, there exists a unital $^*$-homomorphism $\C^k \to \pi_\tau(\M)$ such that the image of each $e_j$ has trace $\frac1k$ and $\|\cdot\|_{2,\tau}$-approximately commutes with $\pi_\tau(\mathcal F)$.

By \cite[Theorem 4.6]{Lo93} (cf.\ \cite[Proposition~2.6]{AP}), the cone over $\C^k$ is a projective C$^*$-algebra. Combined with the structure theorem for order zero maps \cite[Corollary 3.1]{WZ09}, this implies that the constructed $^*$-homomorphism $\C^k \to \pi_\tau(\M)$ lifts to a c.p.c.\ order zero map $\phi_\tau\colon \mathbb{C}^k \rightarrow \M$. Let $\eps > 0$. By continuity, there is a neighbourhood $V_\tau$ of $\tau$ in $K$ such that for all $\sigma \in V_\tau$,
\begin{align}
\label{eq:GammaFibresMain4a}
\|\phi_\tau(1_k)-1_{\M}\|_{2,\sigma} &\leq \frac\eps{m+1}, \\
\label{eq:GammaFibresMain4}
\|[\phi_\tau(e_j),b]\|_{2,\sigma} &\leq \eps, 
\shortintertext{and}
\label{eq:GammaFibresMain5}
 \sigma(\phi_\tau(e_j)^n)&\approx_{\eps/2} 1/k
\end{align}
for all $b \in \mathcal F$, $j\in\{1, \ldots, k\}$, and $n\in\{1, \ldots, N\}$.\footnote{Recall from Section \ref{sec:bundles} that we are identifying the point $\tau \in K$ with the trace $\mathrm{eval}_\tau \circ E$, and the map $\tau \mapsto \mathrm{eval}_\tau \circ E$ is continuous with respect to the weak$^*$ topology on $T(\M)$.}

By compactness and since $\dim(K) \leq m$, we may find a finite $(m+1)$-coloured refinement of $\{V_\tau:\tau \in K\}$---that is, an open cover of $K$ consisting of open sets $U_i^{(c)}$ for $i\in\{1,\dots,l_c\}$ and $c\in\{0,\dots,m\}$ such that each $U_i^{(c)}$ is contained in $V_{\tau_i^{(c)}}$ for some $\tau_i^{(c)}\in K$, and for each $c$, the sets $U_1^{(c)},\dots,U_{l_c}^{(c)}$ are disjoint.
Let $(h_i^{(c)})_{c=0,\dots,m;\,i=1,\dots,l_c}$ be a partition of unity in $C(K) \subseteq \M$ subordinate to this open cover. We now define
\begin{equation}
\label{eq:GammaFibresMain6}
 \Phi^{(c)}= \sum_{i=1}^{l_c} h_i^{(c)}\phi_{\tau_i^{(c)}}\colon \C^k \to \M. \end{equation}
Since each $\phi_{\tau}$ is c.p.c.\ order zero and $h_1^{(c)},\dots,h_{l_c}^{(c)}$ are mutually orthogonal and central, it follows that $\Phi^{(c)}$ is itself a c.p.c.\ order zero map.

To show \eqref{eq:GammaFibresMain2}, we note that for each $c \in \{0,\ldots,m\}$ we have 
\begin{equation}
\begin{split}
\Big\|\Phi^{(c)}(1_k)-\sum_{i=1}^{l_c} h_i^{(c)}\Big\|_{2,\rm u} &\leq
\Big\|\sum_{i=1}^{l_c}(\phi_{\tau_i^{(c)}}(1_k)-1_{\M}) h_i^{(c)}\Big\|_{2,\rm u} \\
&\leq \frac{\eps}{m+1}
\end{split}
\end{equation}
by \eqref{eq:GammaFibresMain4a} since every $\tau \in K$ is in the support of at most one of the functions $h_1^{(c)},\ldots,h_{l_c}^{(c)}$.
Summing over all $c\in\{0,\dots,m\}$ and using the triangle inequality, we obtain \eqref{eq:GammaFibresMain2} since $\sum_{c,i} h_i^{(c)}=1_{\M}$.

To show \eqref{eq:GammaFibresMain1}, fix $c\in\{0,\dots,m\}$, $j\in\{1,\dots,k\}$, and $b \in \mathcal F$.
For $\tau \in K$, there is at most one $i\in\{1,\dots,l_c\}$ such that $\tau \in U_i^{(c)}$.
If no such $i$ exists, then $\big\|[\Phi^{(c)}(e_j),b]\big\|_{2,\tau}=0$, and otherwise, for this $i$, we have
\begin{equation}
\begin{split}
\big\|[\Phi^{(c)}(e_j),b]\big\|_{2,\tau}
&\eqrel[l]{$\stackrel{\eqref{eq:GammaFibresMain6}}{=}$} \big\|[h_i^{(c)}\phi_{\tau_i^{(c)}}(e_j),b]\big\|_{2,\tau} \\
&\eqrel[l]{$\stackrel{\eqref{eq:GammaFibresMain4}}{\leq}$} \eps.
\end{split}
\end{equation}
Similarly for \eqref{eq:GammaFibresMain3}, fix $\tau \in K$, $c\in\{0,\ldots,m\}$, and $n\in\{1,\ldots,N\}$.
If there is no $i$ for which $\tau \in U_i^{(c)}$, then 
\begin{equation}
	\tau(\Phi^{(c)}(e_j)^n)=\tau(\Phi^{(c)}(1_k)^n)=0
\end{equation}
for all $j \in \{1, \ldots, k\}$.
Otherwise, there is exactly one $i$ for which $\tau \in U_i^{(c)}$, and then we have, for $j\in\{1,\dots,k\}$,
\begin{equation}\label{eq:pi}
\begin{split}
\tau(\Phi^{(c)}(e_j)^n)
&\eqrel[i]{$\stackrel{\eqref{eq:GammaFibresMain6}}=$}
\tau((h_i^{(c)})^n\phi_{\tau_i^{(c)}}(e_j)^n) \\
&\eqrel[i]{$=$} \tau((h_i^{(c)})^n)\tau(\phi_{\tau_i^{(c)}}(e_j)^n) \\
&\eqrel[i]{$\stackrel{\eqref{eq:GammaFibresMain5}}{\capprox{\eps/2}}$} \frac1k\tau((h_i^{(c)})^n),
\end{split}
\end{equation}
where the unlabelled inequality uses that the $h_{i}^{(c)}$ are central and $\tau$ is an extremal trace on $\mathcal M$ (as the fibre over $\tau$ is a factor).\footnote{More generally, if $A$ is a C$^*$-algebra, $h, a \in A$ with $h$ central, and $\tau$ is an extremal trace on $A$, then $\tau(ha) = \tau(h) \tau(a)$.  Indeed, we may assume $h \geq 0$.  If $\tau(h) = 0$, this follows from the Cauchy--Schwarz inequality.  When $\tau(h) \neq 0$, note that $b \mapsto \tau(hb)/\tau(h)$ is a trace on $A$ dominated by $\tau(h)^{-1} \tau$ and hence equals $\tau$.}
Since  $\Phi^{(c)}$ is c.p.c.\ order zero, we have $\Phi^{(c)}(1_k)^n = \sum_{j'=1}^k \Phi^{(c)}(e_{j'})^n$.
Thus,
\begin{equation}
\begin{split}
\frac1k\tau(\Phi^{(c)}(1_k)^n)&=\frac1k\sum_{j'=1}^k \tau(\Phi^{(c)}(e_{j'})^n) \\
&\approx_{\eps/2} \frac1k\tau((h_i^{(c)})^n)  \\
&\approx_{\eps/2} \tau(\Phi^{(c)}(e_j)^n)
\end{split}
\end{equation}
for all $j\in\{1,\dots,k\}$, using \eqref{eq:pi} for both approximations.
\end{proof}

\subsection{Orthogonal tracial division}

The next step towards proving Theorem~\ref{thm:GammaFibresBundles} is the construction of mutually orthogonal positive contractions in $\M^\infty$ that commute with a given separable subset $S$, satisfy tracial factorisation with respect to $S$, and do not vanish on any trace in $K^\infty$.
The existence of such families of mutually orthogonal positive contractions follows from Lemma \ref{lem:GammaFibresMain} and makes crucial use of the fact that $m$ is independent of $k$. 

\begin{lemma}[{cf.\ \cite[Lemma 4.3]{TWW15}}]
\label{lem:OrthogonalContractions}
Given $m,r \in \mathbb{N}$, there exists $\gamma_{m,r}>0$ with the following property:
for any $W^*$-bundle $(\M,K,E)$ with $\dim(K)\leq m$ whose fibres are {\rm II}$_1$ factors with property $\Gamma$ and  $\|\cdot\|_{2,\rm u}$-separable subset $S \subseteq \M^\infty$, there exist mutually orthogonal positive contractions $d_0,\dots,d_r\in \M^\infty \cap S'$ such that
\begin{align}
\label{eq:orthog-lemma1a}
\tau(f(d_i)s)&=\tau(f(d_i))\tau(s)
\shortintertext{and}
\label{eq:orthog-lemma1b}
 \tau(d_i) &\geq \gamma_{m,r}
\end{align}
for all $i\in\{0,\dots,r\},\ s\in S,\ \tau \in K^\infty$, and $f \in C_0(0,1]$.
\end{lemma}

\begin{proof}
Fix $m \in \N$. We begin with the case $r = 1$ and define
\begin{equation}\label{eq:gamma-m1}
    \gamma_{m,1} = \frac1{4(m+1)}. 
\end{equation}
Let $S \subseteq \M^\infty$ be $\|\cdot\|_{2,\rm u}$-separable. Let $\M_0 \subseteq \M$ be a $\|\cdot\|_{2,\rm u}$-separable subalgebra of $\M$ such that Lemma \ref{lem:ReindexTracialfactorisation} holds.  

It suffices to prove that there exist orthogonal positive contractions $d_0',d_1'\in \M^\infty \cap \M_0'$ such that
\begin{equation}
\label{eq:OrthogonalContractions1}
    \tau(d_i')\geq \gamma_{m,1}
\end{equation}
for all $\tau \in K^\infty$ and $i\in\{0,1\}$.  Indeed, taking $T$ to be the C$^*$-algebra generated by $d_0'$ and $d_1'$, Lemma \ref{lem:ReindexTracialfactorisation} provides us with a reindexing $^*$-homomorphism $\psi_\rho:\M^\infty \rightarrow \M^\infty$ such that $\psi_\rho(T) \subseteq \M^\infty \cap S'$ and the tracial factorisation in \eqref{eq:ReindexTracialfactorisation} holds.  Define $d_i = \psi_\rho(d_i')$ for $i\in\{0,1\}$ and note that \eqref{eq:orthog-lemma1a} follows from \eqref{eq:ReindexTracialfactorisation}.  Further, for all $\tau \in K^\infty$, we have $\tau \circ \psi_\rho \in K^\infty$, so \eqref{eq:orthog-lemma1b} follows from \eqref{eq:OrthogonalContractions1}.

Let $\Phi^{(0)},\dots,\Phi^{(m)}\colon \C^{2(m+1)} \to \M^\infty \cap \M_0'$ be given by Lemma \ref{lem:GammaFibresMain} (with $k=2(m+1)$ and $S=\M_0$).
Define
\begin{equation} \label{eqn:def-of-a}
    a= \sum_{c=0}^m \Phi^{(c)}(e_1), 
\end{equation}
which is a positive contraction in $\M^\infty \cap \M_0'$ since $\sum_c \Phi^{(c)}$ is a u.c.p.\ map. 
Making use of the continuous functions defined in \eqref{eq:gDef}, set
\begin{equation} 
d_0'= g_{\gamma_{m,1},2\gamma_{m,1}}(a) \quad \text{and} \quad  d_1'= 1_{\M^\infty}-g_{0,\gamma_{m,1}}(a)
\end{equation}
and note that these are orthogonal positive contractions in $\M^\infty \cap \M_0'$ by construction.

To show \eqref{eq:OrthogonalContractions1} for $i=0$, observe that $g_{\gamma_{m,1},2\gamma_{m,1}}(t) \geq t - \gamma_{m,1}$ for all $t \in [0,1]$, and so for $\tau \in K^\infty$,
\begin{equation}
\begin{split}
	\tau(d_0') &\eqrel[b]{$\geq$} \tau(a) - \gamma_{m,1} \\
	&\eqrel[b]{$\stackrel{\eqref{eqn:def-of-a}}{=}$} \sum_{c=0}^m \tau(\Phi^{(c)}(e_1)) - \gamma_{m,1} \\
 &\eqrel[b]{$\stackrel{\eqref{eq:GammaFibresMainB}}{=}$}  \tfrac{1}{2(m+1)} \sum_{c=0}^m \tau(\Phi^{(c)}(1_{2(m+1)})) - \gamma_{m,1} 
 \\
	&\eqrel[b]{$\stackrel{\eqref{eq:GammaFibresMainA}}{=}$} \tfrac{1}{2(m+1)} - \gamma_{m,1} \\
        &\eqrel[b]{$\stackrel{\eqref{eq:gamma-m1}}{=}$} \gamma_{m,1}.
\end{split}
\end{equation}

To show \eqref{eq:OrthogonalContractions1} for $i=1$, we compute that for $\tau \in K^\infty$,
\begin{equation}
\begin{split}
 	\tau(1_{\M^\infty} - d_1') &\eqrel[a]{$=$} \tau(g_{0,\gamma_{m,1}}(a))\\
        &\eqrel[a]{$\leq$} \lim_{l \rightarrow \infty} \tau(a^{1/l})\\
	&\eqrel[a]{$\stackrel{\eqref{eqn:def-of-a}}{=}$} \lim_{l \rightarrow \infty} \tau\Big(\Big(\sum_{c=0}^m\Phi^{(c)}(e_1)\Big)^{1/l}\Big)\\
	&\eqrel[a]{$\leq$} \lim_{l \rightarrow \infty} \sum_{c=0}^m  \tau(\Phi^{(c)}(e_1)^{1/l}) \\
	&\eqrel[a]{$\stackrel{\eqref{eq:GammaFibresMainB}}{=}$} \lim_{l \rightarrow \infty} \sum_{c=0}^m  \tfrac{1}{2(m+1)}\tau(\Phi^{(c)}(1_{2(m+1)})^{1/l})\\
	&\eqrel[a]{$\leq$} \tfrac{m+1}{2(m+1)}\\
	&\eqrel[a]{$=$} \tfrac{1}{2},
\end{split}
\end{equation}
where in the fourth line we use the fact that $\sum_{c=0}^m\Phi^{(c)}(e_1)$ is Cuntz subequivalent to $\bigoplus_{c=0}^m\Phi^{(c)}(e_1)$.\footnote{For positive elements $a$ and $b$ in a $C^*$-algebra $A$, $a$ is \emph{Cuntz subequivalent} to $b$ if there is a sequence $(v_n)_{n=1}^\infty \subseteq A$ with $\|v_n^*bv_n - a \| \rightarrow 0$.  The relevance of this relation dates back to \cite[Section~1]{Cuntz}.}
 Thus, $\tau(d_1') \geq 1/2 \geq\gamma_{m,1}$. This completes the proof of the case $r=1$.

In the general case, by enlarging $r$, we may assume $r = 2^l - 1$ for some $l \geq 1$.  We will prove the result by induction on $l$, starting with the case $l = 1$ handled above.  Fix $m\in\N$ and a $\|\cdot\|_{2,\rm u}$-separable subset $S \subseteq M^\infty$. Assume the result holds for $r = 2^l - 1$ and let $d_0, \ldots, d_r \in \mathcal M^\infty \cap S'$ be positive orthogonal contractions satisfying \eqref{eq:orthog-lemma1a} and \eqref{eq:orthog-lemma1b}.  Let $T \subseteq \mathcal M^\infty$ denote the C$^*$-algebra generated by $S \cup \{ d_0, \ldots, d_r \}$. Note that $T$ is $\|\cdot\|_{2,\rm u}$-separable.  

By the $r=1$ case proved above (but now with $T$ replacing $S$), there are positive orthogonal contractions $\tilde{d}_0, \tilde{d}_1 \in \mathcal M^\infty \cap T'$ satisfying
\begin{align}
	\tau(f(\tilde{d}_i)t)&=\tau(f(\tilde{d}_i))\tau(t)
    \label{eq:orthog-lemma2a}
\shortintertext{and} 
	\tau(\tilde{d}_i) &\geq \gamma_{m,1}
    \label{eq:orthog-lemma2b}
\end{align}
for all $i\in\{0, 1\}$, $t \in T$, $\tau \in K^\infty$, and $f \in C_0((0, 1])$.  We will show the $2^{l+1}$ elements $\tilde d_i d_j \in \mathcal M^\infty \cap S'$ for $i\in\{0, 1\}$ and $j\in\{0, \ldots, r\}$ satisfy the required properties with $\gamma_{m, 2r+1} = \gamma_{m, 1} \gamma_{m, r}$.

First note that the $\tilde d_i d_j$ are clearly mutually orthogonal positive contractions as each $\tilde d_i$ commutes with each $d_j$ by construction.  For all $i\in\{0, 1\}$, $j\in\{0, \ldots, r\}$ and $\tau \in K^\infty$, we have
\begin{equation}
	\tau(\tilde d_i d_j) \stackrel{\eqref{eq:orthog-lemma2a}}{=} \tau(\tilde d_i) \tau(d_j) \geq \gamma_{m, 1} \gamma_{m, r} = \gamma_{m, 2r+1}.
\end{equation}
Let $i \in \{0, 1\}$, $j \in \{0, \ldots, r\}$, $s \in S$, $\tau \in K^\infty$, and $n \in \N$.  Then
\begin{equation}
\begin{split}
	\tau((\tilde d_i d_j)^n s) 
	&\eqrel[c]{$=$} \tau(\tilde d_i^n d_j^n s) \\
	&\eqrel[c]{$\stackrel{\eqref{eq:orthog-lemma2a}}{=}$} \tau(\tilde d_i^n) \tau(d_j^n s) \\
	&\eqrel[c]{$\stackrel{\eqref{eq:orthog-lemma1a}}{=}$} \tau(\tilde d_i^n) \tau(d_j^n) \tau(s)  \\
	&\eqrel[c]{$\stackrel{\eqref{eq:orthog-lemma2a}}{=}$}
	\tau(\tilde d_i^n d_j^n) \tau(s) \\
	&\eqrel[c]{$=$} \tau((\tilde d_i d_j)^n) \tau(s).
\end{split}
\end{equation}
By the Stone--Weierstrass theorem, this implies
\begin{equation}
	\tau(f(\tilde d_i d_j)s) = \tau(f(\tilde d_i d_j)) \tau(s)
\end{equation}
for all $f \in C_0(0,1]$.
\end{proof}

\subsection{Orthogonal gluing}

A sum of c.p.c.\ order zero maps is typically no longer order zero. One way to get an order zero sum is to ensure that the ranges of the maps to be summed are orthogonal.  

We can force this to be the case by multiplying the order zero maps with orthogonal positive contractions commuting with the ranges. 
The following lemma carries out this orthogonal gluing for the c.p.c.\ order zero maps constructed in Lemma \ref{lem:GammaFibresMain} using the orthogonal positive contractions constructed in Lemma \ref{lem:OrthogonalContractions}. 

This gluing operation does not preserve unitality, but we can compute a uniform lower bound for the trace of the image of the unit.  In the next lemma, it is crucial that $\alpha$ depends only on the dimension $m$ of the base space and is independent of both the integer $k$ and the separable subset $S \subseteq \M^\infty$.\footnote{The $\alpha$ constructed is also independent of the bundle $(\mathcal M, K, E)$ itself, but this is not important in the application of the lemma in the proof of Theorem~\ref{thm:GammaFibresBundles}.}

\begin{lemma}[{cf.\ \cite[Proposition 4.4]{TWW15}}]
\label{lem:GammaFibresAlpha}
Given $m \in \mathbb{N}$, there exists $\alpha \in (0,1]$ such that the following holds:
for any $k \in \N$, $W^*$-bundle $(\M,K,E)$ with $\dim(K)\leq m$ whose fibres are {\rm II}$_1$ factors with property~$\Gamma$, and $\|\cdot\|_{2,\rm u}$-separable subset $S \subseteq \M^\infty$, there exists a c.p.c.\ order zero map $\Phi\colon \C^k \to \M^\infty \cap S'$ such that
\begin{align}
	\tau(\Phi(1_k)) &\geq \alpha,
\label{eq:GammaFibresAlphaA}\\
	\tau(f(\Phi(e_j))) &= \frac{1}{k}\tau(f(\Phi(1_k))),
\label{eq:GammaFibresAlphaB}
\shortintertext{and}
	\tau(f(\Phi(e_j))s) &= \tau(f(\Phi(e_j)))\tau(s)
\label{eq:GammaFibresAlphaC}
\end{align}
for all $\tau \in K^\infty,\ j\in\{1,\dots,k\},\ f\in C_0(0,1]$, and $s\in S$.
 \end{lemma}
 
\begin{proof}
Set $\alpha =\gamma_{m,m}$ from Lemma \ref{lem:OrthogonalContractions}. 
Fix $k \in \N$ and let $(\M,K,E)$ be a W$^*$-bundle with $\dim(K)\leq m$ whose fibres are II$_1$ factors with property $\Gamma$.
Further, fix a $\|\cdot\|_{2,\rm u}$-separable subset $S \subseteq \M^\infty$.

Let $\M_0 \subseteq \M$ be the $\|\cdot\|_{2,\rm u}$-separable subset of $\M$ such that Lemma~\ref{lem:ReindexTracialfactorisation} holds.
It suffices to prove that there exists a c.p.c.\ order zero map $\Phi\colon\C^k \to \M^\infty \cap \M_0'$ such that \eqref{eq:GammaFibresAlphaA} and \eqref{eq:GammaFibresAlphaB} hold for all $\tau \in K^\infty,\ j\in\{1,\dots,k\}$ and $f\in C_0(0,1]$. Indeed, taking $T$ to be the C$^*$-algebra generated by $\{\Phi(e_j): j =1,\ldots,k\}$, Lemma \ref{lem:ReindexTracialfactorisation} provides us with a reindexing $^*$-homomorphism $\psi_\rho:\M^\infty \rightarrow \M^\infty$ such that, after replacing $\Phi$ with $\psi_\rho \circ \Phi$, all three conditions \eqref{eq:GammaFibresAlphaA}, \eqref{eq:GammaFibresAlphaB} and \eqref{eq:GammaFibresAlphaC} are satisfied.

Let $\Phi^{(0)},\ldots,\Phi^{(m)}\colon\mathbb{C}^{k}\rightarrow \M \cap \M_0'$ be maps as in Lemma~\ref{lem:GammaFibresMain}.
Let $d_0,\ldots,d_m \in \M^\infty \cap \big(\Phi^{(0)}(\C^k) \cup \cdots \cup \Phi^{(m)}(\C^k) \cup \M_0\big)'$ be orthogonal positive contractions, constructed using Lemma~\ref{lem:OrthogonalContractions}, such that
\begin{align}
\label{eq:GammaFibresAlpha1}
\tau(f(d_i)b)&= \tau(f(d_i))\tau(b)
\shortintertext{and}
\label{eq:GammaFibresAlpha2}
\tau(d_i) &\geq \gamma_{m,m} = \alpha
\end{align}
for all $c\in\{0,\dots,m\}$,  $b \in {\rm C}^*(\Phi^{(c)}(\mathbb C^k))$,  $\tau \in K^\infty$, and $f\in C_0(0,1]$.
Define
\begin{equation}
\label{eq:GammaFibresAlpha3}
 \Phi = \sum_{c=0}^m d_c\Phi^{(c)}\colon\mathbb{C}^k \rightarrow \M^\infty \cap \M_0'. \end{equation}
Since the $d_0,\ldots,d_m$ are orthogonal positive contractions commuting with the images of the $\Phi^{(c)}$, $\Phi$ is c.p.c.\ order zero.

Let $\tau \in K^\infty$, $j\in\{1,\dots,k\}$, and $n\in\N$. Then
\begin{equation}\label{eq:GammaFibresAlphaB-computation}
\begin{split}
	\tau(\Phi(e_j)^n) &\eqrel[e]{$\stackrel{\eqref{eq:GammaFibresAlpha3}}=$} \sum_{c=0}^m \tau(d_c^n\Phi^{(c)}(e_j)^n)\\
	&\eqrel[e]{$\stackrel{\eqref{eq:GammaFibresAlpha1}}{=}$} \sum_{c=0}^m \tau(d_c^n)\tau(\Phi^{(c)}(e_j)^n)\\
	&\eqrel[e]{$\stackrel{\eqref{eq:GammaFibresMainB}}=$} \sum_{c=0}^m \frac{1}{k}\tau(d_c^n)\tau(\Phi^{(c)}(1_k)^n)\\
	&\eqrel[e]{$\stackrel{\eqref{eq:GammaFibresAlpha1}}{=}$}  \frac1k \sum_{c=0}^m \tau(d_c^n\Phi^{(c)}(1_k)^n)\\
    &\eqrel[e]{$\stackrel{\eqref{eq:GammaFibresAlpha3}}=$} \frac{1}{k}\tau(\Phi(1_k)^n),
\end{split}
\end{equation}
using in the first and last lines that the $d_c$ are mutually orthogonal and  each $d_c$ commutes with the range of each $\Phi^{(c')}$.
By linearity, continuity, and the Stone--Weierstrass theorem, \eqref{eq:GammaFibresAlphaB} follows from \eqref{eq:GammaFibresAlphaB-computation}.

Now let $\tau \in K^\infty$. Then
\begin{equation}
\begin{split}
	\tau(\Phi(1_k)) &\eqrel[d]{$\stackrel{\eqref{eq:GammaFibresAlpha3}}=$} \sum_{c=0}^m \tau(d_c\Phi^{(c)}(1_k))\\
 	&\eqrel[d]{$\stackrel{\eqref{eq:GammaFibresAlpha1}}{=}$} \sum_{c=0}^m \tau(d_c)\tau(\Phi^{(c)}(1_k))\\
	&\eqrel[d]{$\stackrel{\eqref{eq:GammaFibresAlpha2}}{\geq}$} \sum_{c=0}^m \alpha\tau(\Phi^{(c)}(1_k)).
\end{split}
\end{equation}
Hence,  $\tau(\Phi(1_k)) \geq \alpha$ by \eqref{eq:GammaFibresMainA}. This verifies \eqref{eq:GammaFibresAlphaA}.
\end{proof}
\subsection{The maximality argument}

Theorem \ref{thm:GammaFibresBundles} is now proven via a maximality argument based on Lemma \ref{lem:GammaFibresAlpha}.
Roughly, if we can take $\alpha = 1$ in Lemma~\ref{lem:GammaFibresAlpha}, then Theorem~\ref{thm:GammaFibresBundles} follows.  
By a reindexing argument, there is a maximal $\alpha_0$ which satisfies Lemma~\ref{lem:GammaFibresAlpha}.  
If $\alpha_0 < 1$, we will use Lemma~\ref{lem:OrthogonalContractions} and the order zero functional calculus (see Section \ref{subsec:order-zero}) to construct a larger $\alpha$ satisfying Lemma~\ref{lem:GammaFibresAlpha}, which will yield a contradiction.

\begin{proof}[Proof of Theorem \ref{thm:GammaFibresBundles}]
Suppose $(\mathcal M, K, E)$ is a W$^*$-bundle such that $m = \dim(K) < \infty$ and every fibre of $\mathcal M$ is a II$_1$ factor with property $\Gamma$.  Let $\Omega$ be the set of all $\alpha \in [0,1]$ for which the conclusion of Lemma~\ref{lem:GammaFibresAlpha} holds, and set $\alpha_0 = \sup \Omega$.
Lemma~\ref{lem:GammaFibresAlpha} implies $\alpha_0 > 0$. Moreover, a standard reindexing argument shows that $\Omega$ is a closed set, so $\alpha_0 \in \Omega$.

It suffices to show that $\alpha_0 = 1$.  Indeed, in this case, for every $k \in \mathbb N$ and $\|\cdot\|_{2, \mathrm u}$-separable subset $S \subseteq \M^\infty$, there exists a c.p.c.\ order zero map $\Phi\colon\C^k \rightarrow \M^\infty \cap S'$ such that
\begin{align}
	\tau(\Phi(1_k)) &\geq 1, 
\label{eq:GammaFibresAlphaEquals1A}\\
	\tau(\Phi(e_j)) &= \frac{1}{k}\tau(\Phi(1_k)),
\label{eq:GammaFibresAlphaEquals1B}
\shortintertext{and}
	\tau(\Phi(e_j)s) &= \tau(\Phi(e_j))\tau(s)
\label{eq:GammaFibresAlphaEquals1C}
\end{align}
for all $\tau \in K^\infty,\ j\in\{1,\dots,k\}$, and $s\in S$. It follows from \eqref{eq:GammaFibresAlphaEquals1A} that $\tau(1_{\M^\infty}-\Phi(1_k)) \leq 0$ for all $\tau \in K^\infty$. Since $1_{\M^\infty} - \Phi(1_k)$ is positive, this implies $\Phi(1_k)=1_{\M^\infty}$. By \cite[Theorem 3.3]{WZ09}, a u.c.p.\ order zero map is a $^*$-homomorphism. Therefore, $\Phi(e_1),\ldots,\Phi(e_k)$ are orthogonal projections summing to $1_{\M^\infty}$, and by \eqref{eq:GammaFibresAlphaEquals1B} and \eqref{eq:GammaFibresAlphaEquals1C}, these projections witness that $\M$ has property $\Gamma$.

Assume for the sake of contradiction that $\alpha_0<1$.  Let $\gamma = \gamma_{m,1} > 0$ be as in Lemma~\ref{lem:OrthogonalContractions}.
Since $0 < \alpha_0 < 1$ and $\gamma > 0$, we may choose $\eps>0$ so that
\begin{equation}\label{eq:GammaFibres-alpha}
\alpha = \alpha_0+\gamma(\alpha_0-\alpha_0^2)-\eps(1-\gamma\alpha_0) > \alpha_0.
\end{equation}
We will show that the conclusion of Lemma~\ref{lem:GammaFibresAlpha} holds with this $\alpha$ (i.e.\ $\alpha \in \Omega$), which will be a contradiction.

Suppose $k \in \mathbb N$ and let $S \subseteq \M^\infty$ be a $\|\cdot\|_{2,\rm u}$-separable subset. Let $\M_0 \subseteq \M$ be a $\|\cdot\|_{2,\rm u}$-separable subalgebra obtained by applying Lemma \ref{lem:ReindexTracialfactorisation} to $S$. By hypothesis, there exists a c.p.c.\ order zero map $\Phi_0\colon\mathbb{C}^k \rightarrow \M^\infty \cap \M_0'$ satisfying
\begin{align}
\label{eq:GammaFibres1}
	\tau(\Phi_0(1_k)) &\geq \alpha_0, \\
\label{eq:GammaFibres2}
	\tau(f(\Phi_0(e_j))) &= \frac{1}{k}\tau(f(\Phi_0(1_k))),
\shortintertext{and}
\label{eq:GammaFibres3}
	\tau(f(\Phi_0(e_j))b) &= \tau(f(\Phi_0(e_j)))\tau(b)
\end{align}
for all $\tau \in K^\infty,\ j\in\{1,\dots,k\},\ f\in C_0(0,1]$, and $b\in \M_0$.
Using Lemma~\ref{lem:OrthogonalContractions}, let $d_0,d_1 \in \M^\infty \cap (\M_0 \cup \Phi_0(\C^k))'$ be orthogonal positive contractions such that
\begin{align}
\label{eq:GammaFibres4}
\tau(f(d_i)b)&= \tau(f(d_i))\tau(b)
\shortintertext{and}
\label{eq:GammaFibres5}
\tau(d_i) &\geq \gamma
\end{align}
for all $i\in\{0,1\},\ \tau \in K^\infty$, $f \in C_0(0,1]$, and $b \in {\rm C}^*(\Phi_0(\mathbb C^k))$.

Let $g_{0,\eps},g_{\eps,2\eps} \in C_0(0,1]$ be the continuous functions defined in \eqref{eq:gDef} and set $\Delta_\eps = g_{0,\eps} - g_{\eps,2\eps}$.
Using the order zero functional calculus, define $\Phi_0' \colon\C^k \to \M^\infty \cap \M_0'$ by
\begin{equation}\label{eq:GammaFibres6}
\begin{split}
\Phi_0' &= d_0\Delta_\eps(\Phi_0)+g_{\eps,2\eps}(\Phi_0) \\
&= d_0 g_{0, \eps}(\Phi_0) + (1 - d_0) g_{\eps, 2\eps}(\Phi_0).
\end{split}
\end{equation}
Since $d_0$ commutes with ${\rm C}^*(\Phi_0(\mathbb C^k))$ and $g_{\eps, 2\eps} \leq g_{0, \eps}$, we have that $\Phi_0' \leq g_{0, \eps}(\Phi_0)$.  Since $g_{0, \eps}(\Phi_0)$ is c.p.c.\ order zero, so is $\Phi_0'$ by Lemma~\ref{lem:oz-leq}.
Fix $\tau \in K^\infty$, $j \in \{1,\dots,k\}$, and $n \in \N$. Using Lemma~\ref{lem:oz-projection} and the binomial theorem, we compute that
\begin{equation}
\begin{split}
\tau(\Phi_0'(e_j)^n)
&\eqrel[f]{$\stackrel{\eqref{eq:GammaFibres6}}=$} \sum_{i=0}^n \binom ni \tau\big(d_0^i(\Delta_\eps^ig_{\eps,2\eps}^{n-i})(\Phi_0(e_j))\big) \\
&\eqrel[f]{$\stackrel{\eqref{eq:GammaFibres4}}=$} \sum_{i=0}^n \binom ni \tau(d_0^i)\tau\big((\Delta_\eps^ig_{\eps,2\eps}^{n-i})(\Phi_0(e_j))\big) \\
&\eqrel[f]{$\stackrel{\eqref{eq:GammaFibres2}}=$} \frac1k \sum_{i=0}^n \binom ni  \tau(d_0^i)\tau\big((\Delta_\eps^ig_{\eps,2\eps}^{n-i})(\Phi_0(1_k))\big) \label{eq:GammaFibres7} \\
&\eqrel[f]{$\stackrel{\eqref{eq:GammaFibres4}}{=}$} \frac1k \sum_{i=0}^n \binom ni \tau\big(d_0^i  (\Delta_\eps^ig_{\eps, 2\eps}^{n -i})(\Phi_0(1_k))\big) \\
&\eqrel[f]{$\stackrel{\eqref{eq:GammaFibres6}}{=}$}
\frac1k \tau(\Phi_0'(1_k)^n).
\end{split}
\end{equation}

Next, we define the positive contraction\footnote{To see that $h$ commutes with the range of $\Phi_0$, note that the structure theorem for order zero maps implies $\Phi_0(1_k)$ commutes with the range of $\Phi_0$.}
\begin{equation}
\label{eq:GammaFibres8}
h = d_1(1_{M^\infty}-g_{0,\eps}(\Phi_0(1_k))) \in \M^\infty \cap (\M_0 \cup \Phi_0(\C^k))'. \end{equation}
Since $d_0\perp d_1$ and $(1_{\M^\infty}-g_{0,\eps}(\Phi_0(1_k)))\perp g_{\eps,2\eps}(\Phi_0(1_k))$, we see that
\begin{equation}
\label{eq:GammaFibres9}
 h \perp d_0\Delta_\eps(\Phi_0(1_k))+g_{\eps,2\eps}(\Phi_0(1_k)) \stackrel{\eqref{eq:GammaFibres6}}{=}\Phi_0'(1_k).
\end{equation}
Using again that $\alpha_0$ satisfies Lemma~\ref{lem:GammaFibresAlpha}, there is a c.p.c.\ order zero map $\Phi_1\colon\mathbb{C}^k \rightarrow \M^\infty \cap (\M_0 \cup \{h\})'$ satisfying
\begin{align}
\label{eq:GammaFibres10}
	\tau(\Phi_1(1_k)) &\geq \alpha_0, \\
\label{eq:GammaFibres11}
	\tau(f(\Phi_1(e_j))) &= \frac{1}{k}\tau(f(\Phi_1(1_k))),
\shortintertext{and}
\label{eq:GammaFibres12}
	\tau(f(\Phi_1(x))b) &= \tau(f(\Phi_1(x)))\tau(b)
\end{align}
for all $\tau \in K^\infty,\ x \in \mathbb C^k,\ f\in C_0(0,1]$, and $b \in {\rm C}^*(h)$.
Now, define
\begin{equation}
\label{eq:GammaFibres13}
\Phi=\Phi_0'+h\Phi_1\colon\C^k \to \M^\infty \cap \M_0'.
\end{equation}
Since $h$ commutes with the range of $\Phi_1$, $h\Phi_1$ is c.p.c.\ order zero.  By \eqref{eq:GammaFibres9}, $h$ is orthogonal to $\Phi_0'(1_k)$, and using the structure theorem for order zero maps, $h$ is also orthogonal to the range of $\Phi_0'$.  So $\Phi$ is a sum of c.p.c.\ order zero maps with orthogonal ranges and hence is itself a c.p.c.\ order zero map.

We shall show that $\Phi$ satisfies both \eqref{eq:GammaFibresAlphaA} and \eqref{eq:GammaFibresAlphaB}. First, we show \eqref{eq:GammaFibresAlphaB}.
For $\tau \in K^\infty$, $j\in\{1,\dots,k\}$, and $n \in \N$, we have
\begin{equation}
\begin{split}
\tau(\Phi(e_j)^n)
&\eqrel[g]{$\stackrel{\eqref{eq:GammaFibres13}}=$}
\tau(\Phi_0'(e_j)^n+h^n\Phi_1(e_j)^n) \\
&\eqrel[g]{$\stackrel{\eqref{eq:GammaFibres12}}=$} \tau(\Phi_0'(e_j)^n)+\tau(h^n)\tau(\Phi_1(e_j)^n) \\
&\eqrel[g]{$\stackrel{\eqref{eq:GammaFibres7},\eqref{eq:GammaFibres11}}=$} \frac1k\tau(\Phi_0'(1_k)^n)+\frac1k\tau(h^n)\tau(\Phi_1(1_k)^n) \\
&\eqrel[g]{$\stackrel{\eqref{eq:GammaFibres12},\eqref{eq:GammaFibres13}}=$} \frac1k\tau(\Phi(1_k)^n).
\end{split}
\end{equation}
By linearity, continuity, and the Stone--Weierstrass theorem, \eqref{eq:GammaFibresAlphaB} follows.

We now work towards showing that $\Phi$ satisfies \eqref{eq:GammaFibresAlphaA}. 
Let $\tau \in K^\infty$.
By \eqref{eq:GammaFibres6}, \eqref{eq:GammaFibres13}, and Lemma~\ref{lem:oz-projection}, we have
\begin{equation}\label{eq:sam-1}
\tau(\Phi(1_k))
= \tau(d_0\Delta_\eps(\Phi_0(1_k)))+\tau(g_{\eps,2\eps}(\Phi_0(1_k)))+\tau(h\Phi_1(1_k)). \\
\end{equation}
We estimate the first term of \eqref{eq:sam-1} as follows:
\begin{equation}\label{eq:dag}
\begin{split}
\tau\big(d_0\Delta_\eps(\Phi_0(1_k))\big)
&\eqrel[j]{$\stackrel{\eqref{eq:GammaFibres4}}{=}$}
\tau(d_0) \tau\big(\Delta_\eps(\Phi_0(1_k))\big) \\
&\eqrel[j]{$\stackrel{\eqref{eq:GammaFibres5}}{\geq}$}
\gamma \tau\big(\Delta_\eps(\Phi_0(1_k))\big) \\
&\eqrel[j]{$\geq$}
\gamma\alpha_0 \tau\big(\Delta_\eps(\Phi_0(1_k))\big),
\end{split}
\end{equation}
where we have used that $\alpha_0 < 1$ in the last line.  We estimate the third term of \eqref{eq:sam-1} as follows:
\begin{equation}\label{eq:ddag}
\begin{split}
\tau(h\Phi_1(1_k))
&\eqrel[k]{$\stackrel{\eqref{eq:GammaFibres12}}{=}$}
\tau(h)\tau(\Phi_1(1_k)) \\
&\eqrel[k]{$\stackrel{\eqref{eq:GammaFibres10}}{\geq}$} \alpha_0 \tau(h) \\
&\eqrel[k]{$\stackrel{\eqref{eq:GammaFibres8}}{=}$}
\alpha_0 \tau\big(d_1(1_{\mathcal M^\infty} - g_{0, \eps}(\Phi_0(1_k)))\big) \\
&\eqrel[k]{$\stackrel{\eqref{eq:GammaFibres4}}{=}$}
\alpha_0 \tau(d_1)\tau\big(1_{\mathcal M^\infty} - g_{0, \eps}(\Phi_0(1_k))\big) \\
&\eqrel[k]{$\stackrel{\eqref{eq:GammaFibres5}}{\geq}$} \gamma\alpha_0 \tau\big(1_{\mathcal M^\infty} - g_{0, \eps}(\Phi_0(1_k))\big).
\end{split}
\end{equation}
Substituting the estimates \eqref{eq:dag} and \eqref{eq:ddag} into \eqref{eq:sam-1} and using that $\Delta_\eps = g_{0, \eps} - g_{\eps, 2\eps}$, we obtain
\begin{equation}\label{eq:sam-2}
\begin{split}
\tau(\Phi(1_k)) &\geq \gamma\alpha_0 \tau\big(\Delta_\eps(\Phi_0(1_k))\big) + \tau\big(g_{\eps,2\eps}(\Phi_0(1_k))\big) \\
&\qquad + \gamma\alpha_0 \tau\big(1_{\mathcal M^\infty} - g_{0, \eps}(\Phi_0(1_k))\big)\\
&= \gamma\alpha_0 \tau\big(g_{0, \eps}(\Phi_0(1_k))\big) - \gamma\alpha_0 \tau\big(g_{\eps, 2\eps}(\Phi_0(1_k))\big)  \\
&\qquad + \tau\big(g_{\eps,2\eps}(\Phi_0(1_k))\big) + \gamma\alpha_0 - \gamma\alpha_0\tau\big(g_{0, \eps}(\Phi_0(1_k))\big)\\
&= \gamma\alpha_0 + (1 - \gamma\alpha_0) \tau\big(g_{\eps,2\eps}(\Phi_0(1_k))\big).
\end{split}
\end{equation}
Finally, using that $g_{\eps, 2\eps}(t) \geq t - \eps$ for all $t \in [0, 1]$, we get
\begin{equation}
\begin{split}
\tau(\Phi(1_k))
&\eqrel[h]{$\geq$}
\gamma \alpha_0 +(1 - \gamma \alpha_0) \tau\big( \Phi_0(1_k) - \eps 1_{\mathcal M^\infty}\big) \\
&\eqrel[h]{$\stackrel{\eqref{eq:GammaFibres1}}{\geq}$}
\gamma \alpha_0 + (1 - \gamma \alpha_0)(\alpha_0 - \eps) \\ 
&\eqrel[h]{$\stackrel{\eqref{eq:GammaFibres-alpha}}{=}$} \alpha,
\end{split}
\end{equation}
which completes the proof that \eqref{eq:GammaFibresAlphaA} holds. 

By Lemma~\ref{lem:ReindexTracialfactorisation} with $T = {\rm C}^*(\Phi(\C^k)) \subseteq \M^\infty \cap \M_0'$, there exists a strictly increasing function $\rho \colon\N \to \N$ such that, after replacing $\Phi$ with $\psi_\rho\circ\Phi\colon\C^k \rightarrow \M^\infty \cap S'$, we have that \eqref{eq:GammaFibresAlphaC} holds for all $\tau \in K^\infty,\ j\in\{1,\dots,k\},\ f\in C_0(0,1]$, and $b\in S$.

From the definition of the reduced power $(\M^\infty, K^\infty, E^\infty)$, we have $E^\infty \circ \psi_\rho = \psi_\rho \circ E^\infty$. Hence, it follows that \eqref{eq:GammaFibresAlphaA} and \eqref{eq:GammaFibresAlphaB} continue to hold with $\psi_\rho\circ\Phi$ in place of $\Phi$.
Thus we have shown $\alpha \in \Omega$, which is our intended contradiction.
\end{proof}

Restricting to the case of tracial completions of C$^*$-algebras with a Bauer simplex of traces, we obtain Theorem~\ref{thm:GammaFibres}.
\begin{proof}[Proof of Theorem~\ref{thm:GammaFibres}]
    Let $A$ be a C$^*$-algebra with $\partial_e T(A)$ non-empty and compact. 
    If $A$ has uniform property $\Gamma$, then it is clear that $\pi_\tau(A)''$ has property $\Gamma$ for each $\tau \in \partial_e T(A)$.
    
    Suppose now that $K = \partial_e T(A)$ has finite covering dimension and $\pi_\tau(A)''$ has property $\Gamma$ for each $\tau \in K$.
    By \cite[Theorem 3]{Oz13}, the uniform tracial completion $\M=\completion{A}{T(A)}$ has the structure of a W$^*$-bundle over $K$ with fibres $\pi_\tau(A)''$ for $\tau \in K$. 
    By Theorem \ref{thm:GammaFibresBundles}, $\M$ has property $\Gamma$. Hence, by $\|\cdot\|_{2,\rm u}$-density, it follows that $A$ has uniform property~$\Gamma$ (see Section~\ref{subsec:property-gamma}).
\end{proof}

\begin{remark}\label{rem:McDuffFibres}
If $(\M,K,E)$ is a W$^*$-bundle with factorial fibres where $K$ has finite covering dimension and $\pi_\tau(\M)''$ is McDuff for every $\tau \in K$, then $\M$ is McDuff.  The proof follows as in the proof of Theorem~\ref{thm:GammaFibresBundles} except with the order zero maps $\mathbb C^k \rightarrow \mathcal M^\infty \cap S'$ replaced with order zero maps $M_k \rightarrow \mathcal M^\infty \cap \M'$ throughout the proof.  This is essentially the proof of \cite[Theorem~4.6]{TWW15}.
\end{remark}


\begin{thebibliography}{10}

\bibitem{AP}
C.~A.~Akemann and G.~K.~Pedersen.
\newblock Ideal perturbations of elements in C$^*$-algebras.
\newblock {\em Math. Scand.}, 41(1):117--139, 1977.

\bibitem{BK}
E.~Blanchard and E.~Kirchberg.
\newblock Non-simple purely infinite {C$^*$} algebras:
the {H}ausdorff case.
\newblock {\em J. Funct. Anal.}, 207(2):461--513, 2004.

\bibitem{BBSTWW}
J.~Bosa, N.~P.~Brown, Y.~Sato, A.~Tikuisis, S.~White, and W.~Winter.
\newblock Covering dimension of {C$^*$}-algebras and 2-coloured classification.
\newblock {\em Mem. Amer. Math. Soc.}, 257(1233):vii+97, 2019.

\bibitem{Br11}
N.~P.~Brown.
\newblock The symbiosis of {C$^*$}- and {W$^*$}-algebras.
\newblock In {\em Aspects of operator algebras and applications}, volume 534 of
  {\em Contemp. Math.}, pages 121--155. Amer. Math. Soc., Providence, RI, 2011.

\bibitem{TraciallyComplete}
J.~Carrión, J.~Castillejos, S.~Evington, J.~Gabe, C.~Schafhauser, A.~Tikuisis,
  and S.~White.
\newblock Tracially complete {C$^*$}-algebras.
\newblock arXiv:2310.20594.

\bibitem{CE}
J.~Castillejos and S.~Evington.
\newblock Nuclear dimension of simple stably projectionless {${\rm
  C}^*$}-algebras.
\newblock {\em Anal. PDE}, 13(7):2205--2240, 2020.

\bibitem{CETW-classification}
J.~Castillejos, S.~Evington, A.~Tikuisis, and S.~White.
\newblock Classifying maps into uniform tracial sequence algebras.
\newblock {\em M\"unster J. Math.}, 14(2):265--281, 2021.

\bibitem{CETW}
J.~Castillejos, S.~Evington, A.~Tikuisis, and S.~White.
\newblock Uniform property {$\Gamma$}.
\newblock {\em Int. Math. Res. Not. IMRN}, 2022(13):9864--9908, 2022.

\bibitem{CETWW}
J.~Castillejos, S.~Evington, A.~Tikuisis, S.~White, and W.~Winter.
\newblock Nuclear dimension of simple {C$^*$}-algebras.
\newblock {\em Invent. Math.}, 224(1):245--290, 2021.

\bibitem{Co76}
A.~Connes.
\newblock Classification of injective factors. {C}ases {II$_{1}$,} {II$_{\infty
  }$,} {III$_{\lambda }$,} {$\lambda \not=1$}.
\newblock {\em Ann. of Math. (2)}, 104(1):73--115, 1976.

\bibitem{Cuntz}
J.~Cuntz.
\newblock Dimension functions on simple {$C^*$}-algebras.
\newblock {\em Math. Ann.}, 233(2):145--153, 1978.

\bibitem{Di69}
J.~Dixmier.
\newblock Quelques propri\'et\'es des suites centrales dans les facteurs de
  type {${\rm II}_{1}$}.
\newblock {\em Invent. Math.}, 7:215--225, 1969.

\bibitem{Di77}
J.~Dixmier.
\newblock {\em Les {$C^*$}-algebres et leurs repr{\'e}sentations
  (Gauthier-Villars, Paris, 1969); {E}nglish translation: {$C^*$}-Algebras}.
\newblock North-Holland, Amsterdam, New York and Oxford, 1977.

\bibitem{Ev18}
S.~Evington.
\newblock {\em {$W^*$}-Bundles}.
\newblock PhD thesis, University of Glasgow, 2018.
\newblock http://theses.gla.ac.uk/8650/.

\bibitem{JS99}
X.~Jiang and H.~Su.
\newblock On a simple unital projectionless {$\rm{C}^*$}-algebra.
\newblock {\em Amer. J. Math.}, 121(2):359--413, 1999.

\bibitem{KR14}
E.~Kirchberg and M.~R{\o}rdam.
\newblock Central sequence {C$^*$}-algebras and tensorial absorption of the
  {J}iang--{S}u algebra.
\newblock {\em J. Reine Angew. Math.}, 695:175--214, 2014.

\bibitem{Lo93}
T.~A.~Loring.
\newblock {C$^*$}-algebras generated by stable relations.
\newblock {\em J. Funct. Anal.}, 112(1):159--203, 1993.

\bibitem{MS12}
H.~Matui and Y.~Sato.
\newblock Strict comparison and {$\mathcal{Z}$}-absorption of nuclear
  {C$^*$}-algebras.
\newblock {\em Acta Math.}, 209(1):179--196, 2012.

\bibitem{MS14}
H.~Matui and Y.~Sato.
\newblock Decomposition rank of {UHF}-absorbing {$C^*$}-algebras.
\newblock {\em Duke Math. J.}, 163(14):2687--2708, 2014.

\bibitem{MvN43}
F.~J.~Murray and J.~von Neumann.
\newblock On rings of operators. {IV}.
\newblock {\em Ann. of Math. (2)}, 44:716--808, 1943.

\bibitem{Oz13}
N.~Ozawa.
\newblock Dixmier approximation and symmetric amenability for {$\rm
  C^*$}-algebras.
\newblock {\em J. Math. Sci. Univ. Tokyo}, 20(3):349--374, 2013.

\bibitem{Sa12}
Y.~Sato.
\newblock Trace spaces of simple nuclear {C$^*$}-algebras with
  finite-dimensional extreme boundary.
\newblock arXiv:1209.3000.

\bibitem{SWW15}
Y.~Sato, S.~White, and W.~Winter.
\newblock Nuclear dimension and {$\mathcal{Z}$}-stability.
\newblock {\em Invent. Math.}, 202(2):893--921, 2015.

\bibitem{Schaf20}
C.~Schafhauser.
\newblock Subalgebras of simple {AF}-algebras.
\newblock {\em Ann. of Math. (2)}, 192(2):309--352, 2020.

\bibitem{TWW17}
A.~Tikuisis, S.~White, and W.~Winter.
\newblock Quasidiagonality of nuclear {C$^*$}-algebras.
\newblock {\em Ann. of Math. (2)}, 185(1):229--284, 2017.

\bibitem{TWW15}
A.~S.~Toms, S.~White, and W.~Winter.
\newblock {$\mathcal{Z}$}-stability and finite-dimensional tracial boundaries.
\newblock {\em Int. Math. Res. Not. IMRN}, 2015(10):2702--2727, 2015.

\bibitem{Wi12}
W.~Winter.
\newblock Nuclear dimension and {$\mathcal{Z}$}-stability of pure
  {C$^*$}-algebras.
\newblock {\em Invent. Math.}, 187(2):259--342, 2012.

\bibitem{WZ09}
W.~Winter and J.~Zacharias.
\newblock Completely positive maps of order zero.
\newblock {\em M\"unster J. Math.}, 2:311--324, 2009.

\bibitem{W}
M.~Wolff.
\newblock Disjointness preserving operators on {C$^*$}-algebras.
\newblock {\em Arch. Math.}, 62(3):248--253, 1994.



\end{thebibliography}
\end{document}